\documentclass{article}%
\pdfoutput=1
\usepackage{amsmath}
\usepackage{amsfonts}
\usepackage{amssymb}
\usepackage{graphicx}%
\providecommand{\U}[1]{\protect\rule{.1in}{.1in}}
\newtheorem{theorem}{Theorem}

\newtheorem{corollary}[theorem]{Corollary}

\newtheorem{definition}[theorem]{Definition}

\newtheorem{proposition}[theorem]{Proposition}

\newenvironment{proof}[1][Proof]{\noindent\textbf{#1.} }{\ \rule{0.5em}{0.5em}}
\begin{document}

\title{Link colorings and the Goeritz matrix}
\author{Lorenzo Traldi\\Lafayette College\\Easton, Pennsylvania 18042}
\date{}
\maketitle

\begin{abstract}
We discuss the connection between colorings of a link diagram and the Goeritz matrix.

Keywords. Coloring, Goeritz matrix, link

Mathematics Subject Classification. 57M25, 05C10, 05C22, 05C50

\end{abstract}

\section{Introduction}

This paper is inspired by two papers that have appeared previously, the first written by Nanyes \cite{N} and the second written by Lamey, Silver and Williams \cite{LSW}. These papers involve the connection between two ideas from classical knot theory, the Goeritz matrix and colorings of link diagrams. Goeritz introduced his matrix in 1933 \cite{G}, and it was also discussed in Reidemeister's classic treatise \cite{R}. The Goeritz matrix has attracted the attention of many researchers over the decades; see \cite{CG, C, GL, H, ILL, J, K, Li, LW, Lip, P, S, T} for instance. Link colorings were mentioned in the textbook of Crowell and Fox \cite[Exercises VI. 6 and VI. 7]{CF}. Link colorings can be defined easily and they provide very simple nontriviality proofs for some knots and links, so it is natural that they are mentioned in many introductory discussions of knot theory, like \cite{A, CSW, Kau, L}.

Crowell and Fox used link colorings for the purpose of providing
combinatorial descriptions of certain kinds of representations of link groups (the fundamental groups of link complements in $\mathbb{S}^{3}$). At
first glance, this purpose does not suggest a connection with the Goeritz
matrix; link groups are nonabelian in general, and we would expect a matrix of
integers to be associated with abelian groups instead. Nanyes \cite{N}
provided an indirect connection: link colorings with values in an abelian
group $A$ are connected with representations of link groups in a semidirect
product of $A$ and $\mathbb{Z}/2\mathbb{Z}$, and these representations in turn
are connected with the Goeritz matrix. Nanyes's discussion is quite general; it applies to any link diagram, and any abelian group. His presentation requires the theory connecting group representations to covering spaces, and the fact that the Goeritz matrix is associated with 2-fold coverings of $\mathbb{S}^{3}$ branched over links \cite{K, Li, S}. 

Extending a theme established earlier by Kauffman \cite{Krem} and Carter,
Silver and Williams \cite{CSW}, Lamey, Silver and Williams \cite{LSW} showed
that the colorings introduced by Crowell and Fox are related to other types of link colorings, and they
observed that one of these other types of link colorings is directly related to the Goeritz matrix. Unlike Nanyes, they restricted their attention to colorings with values in a field, and to alternating link diagrams for which one of the associated checkerboard graphs is connected. This portion of their paper does not require results from algebraic topology; the arguments involve only relatively elementary properties of matrices and plane graphs.

Taken together, the papers of Lamey, Silver and Williams \cite{LSW} and Nanyes \cite{N} suggest a problem of exposition: to provide an explanation of the connection between link colorings and the Goeritz matrix that is as direct and elementary as the discussion in \cite{LSW}, and as general as the discussion in \cite{N}. The purpose of the present paper is to provide such an explanation. 

Before starting, we should thank an anonymous reviewer whose good advice improved the exposition in several regards.

\section{Link colorings}

We use link diagrams to represent links in the usual way. A tame, classical
link diagram\ begins with a finite number of piecewise smooth, simple closed
curves $\gamma_{1},\dots,\gamma_{\mu}$ in the plane. The only
(self-)intersections of these curves are transverse double points, called
\emph{crossings}; and there are only finitely many of these. At each crossing
a very short segment of one of the incident curves is removed. The resulting
piecewise smooth 1-dimensional subset of $\mathbb{R}^{2}$ is a link diagram.
The set of arc components of a link diagram $D$ is denoted $A(D)$. The
\emph{faces} of $D$ are the arc components of $\mathbb{R}^{2}-\cup\gamma_{i}$,
and the set of faces of $D$ is denoted $F(D)$. A link diagram $D$ represents a
link $L(D)$ in $\mathbb{R}^{3}$, which consists of piecewise smooth, simple
closed curves $K_{1},\dots,K_{\mu}$ such that $K_{i}$ projects to $\gamma_{i}$
for each $i$. $K_{1},\dots,K_{\mu}$ are the \emph{components} of $L(D)$. The
removal of segments in $D$ is used to distinguish the underpassing arc at each crossing.

A link diagram is \emph{split} if the union $\cup \gamma_i$ is not connected. In keeping with the goal of generality mentioned in the introduction, our discussion includes split diagrams as well as non-split diagrams.

We use $A$ to denote an arbitrary abelian group. To avoid notational confusion with the
arcs of a link diagram, we usually use $\alpha$ to represent an element of $A$.

\begin{definition}
\label{foxc}If $D$ is a link diagram then a \emph{Fox coloring} of $D$ with
values in an abelian group $A$ is a mapping $f:A(D)\rightarrow A$ with the following property.

\begin{itemize}
\item If there is a crossing of $D$ at which the underpassing arcs are $a_{1}$
and $a_{2}$ and the overpassing arc is $a_{3}$, then $f(a_{1})+f(a_{2})=2\cdot
f(a_{3})$.
\end{itemize}
The set of Fox colorings of $D$ with values in $A$ is denoted $\mathcal{F}_{A}(D)$.
\end{definition}

Notice that we do not require a Fox coloring of $D$ to be $0$ on any arc of $D$. This generality gives $\mathcal{F}_{A}(D)$ a couple of pleasant (and obvious) naturality properties, which do not hold in \cite{CSW, LSW, N}. Suppose $D$ is the split union of subdiagrams $D_1$ and $D_2$, i.e., $D=D_1 \cup D_2$ and no crossing of $D$ involves both $D_1$ and $D_2$. Then the union of functions defines a bijective map
\[
\mathcal{F}_{A}(D_1) \times \mathcal{F}_{A}(D_2) \to \mathcal{F}_{A}(D)
\]
for every abelian group $A$, and an injective map
\[
\mathcal{F}_{A_{1}}(D_1) \times \mathcal{F}_{A_{2}}(D_2) \to \mathcal{F}_{A_{1} \oplus A_{2}}(D)
\]
for every pair of abelian groups $A_1,A_2$.

\begin{definition}
\label{dehnc}Let $D$ be a link diagram. Then a \emph{Dehn
coloring} of $D$ with values in an abelian group $A$ is a mapping
$d:F(D)\rightarrow A$ with the following property.

\begin{itemize}
\item Suppose $F$ and $F^{\prime}$ are two faces of $D$ whose boundaries share
a segment of positive length, contained in an arc $a\in A(D)$. Then the sum
$d(F)+d(F^{\prime})$ depends only on $a$.
\end{itemize}
The set of Dehn colorings of $D$ with values in $A$ is denoted $\mathcal{D}_{A}(D)$.
\end{definition}

Definition~\ref{dehnc} gives rise to one equation for each crossing of $D$. If the faces incident at a crossing are indexed as in Figure \ref{colorf1} then the boundaries of $F_1$ and $F_4$ share a segment of positive length contained in $a_3$, and so do the boundaries of $F_2$ and $F_3$. Consequently Definition~\ref{dehnc} requires $d(F_{1})+d(F_{4})=d(F_{2})+d(F_{3})$.

\begin{figure}[ht]%
\centering
\includegraphics[
trim=2.139494in 7.633484in 1.605895in 1.344884in,
height=1.548in,
width=3.5933in
]%
{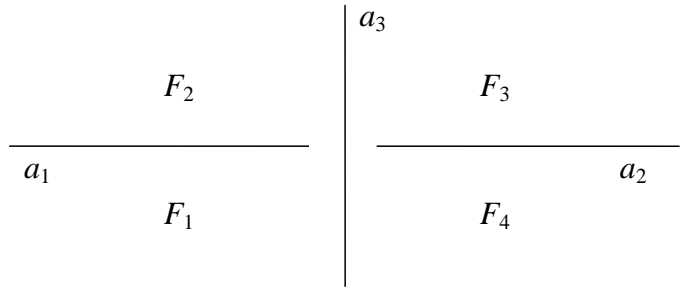}%
\caption{The arcs and faces incident at a crossing.}%
\label{colorf1}%
\end{figure}

We should remark that the term \textquotedblleft Dehn
coloring\textquotedblright\ indicates a courteous regard for one of the important early contributors to combinatorial group theory and geometric topology, but it does not indicate that these colorings were actually introduced by Dehn. Definition \ref{dehnc} was mentioned by Kauffman \cite{Krem} and developed further by Carter,\ Silver and Williams \cite{CSW}, who chose the name \textquotedblleft Dehn
coloring\textquotedblright\ because these colorings are connected with a way to present link groups that was introduced by Dehn. We should also remark that like Definition \ref{foxc}, Definition \ref{dehnc} is generalized from \cite{CSW, Krem} -- we allow $A$ to be an arbitrary abelian group, and
we do not require any value of a Dehn coloring to be $0$. 

As discussed in \cite{CSW} and \cite{Krem}, Dehn colorings and Fox colorings are closely related to each other.

\begin{definition}
\label{fdef}Let $D$ be a link diagram, and $A$ an abelian group. Then $\mathcal{D}_{A}(D)$ and $\mathcal{F}_{A}(D)$ are both abelian groups under pointwise addition of functions. There is a homomorphism $\varphi:\mathcal{D}_{A}%
(D)\rightarrow\mathcal{F}_{A}(D)$, defined by: if $d\in\mathcal{D}%
_{A}(D)$ then
\[
\varphi(d)(a)=d(F_{1})+d(F_{2})
\]
whenever $F_{1},F_{2}$ are two faces of $F$ whose boundaries share a segment
of positive length on $a$.
\end{definition}

Definition \ref{dehnc} implies that $\varphi(d)$ is well defined, and also
that $\varphi(d)$ satisfies Definition \ref{foxc}: with arcs and faces indexed
as in Figure \ref{colorf1},
\begin{align*}
\varphi(d)(a_{1})+\varphi(d)(a_{2})  &  =d(F_{1})+d(F_{2})+d(F_{3})+d(F_{4})\\
&  =(d(F_{1})+d(F_{4}))+(d(F_{2})+d(F_{3}))=2\cdot\varphi(d)(a_{3}).
\end{align*}

As is well known, the faces of a link diagram can be colored in a checkerboard fashion, so that whenever the boundaries of two faces share a segment of positive length, one face is shaded and the other is not shaded. It is traditional to prefer one of the two possible checkerboard shadings, by specifying whether the unbounded face should be shaded or unshaded. In keeping with our theme of generality, however, we do not follow this tradition. The gain in generality is vacuous at this point, but later it will mean that the two different shadings of a link diagram give rise to two different Goeritz matrices. The theory we develop will apply equally well to both matrices.

Arbitrarily choose one of the two checkerboard shadings of a link diagram $D$, and let $\sigma:F(D)\rightarrow\{0,1\}$ be the map defined as follows.
\[
\sigma(F)=%
\begin{cases}
0\text{,} & \text{if }F\text{ is unshaded }\\
1\text{,} & \text{if }F\text{ is shaded }%
\end{cases}
\]
If $\alpha,\beta\in A$, then $D$ has a Dehn coloring
$d_{\alpha,\beta}$ given by%

\[
d_{\alpha,\beta}(F)=(1-\sigma(F))\cdot\alpha+\sigma(F)\cdot\beta\text{.}%
\]
This mapping satisfies Definition \ref{dehnc} because in Figure \ref{colorf1}
each of the pairs $\{F_{1},F_{4}\}$, $\{F_{2},F_{3}\}$ includes one shaded
face and one unshaded face, so that
\[
d_{\alpha,\beta}(F_{1})+d_{\alpha,\beta}(F_{4})=\alpha+\beta=d_{\alpha,\beta
}(F_{2})+d_{\alpha,\beta}(F_{3}).
\]

The next result is our version of \cite[Theorem 2.2]{CSW}. We have an epimorphism rather than an isomorphism because our Dehn colorings are not
required to be $0$ anywhere. The proof is essentially the
same as in \cite{CSW}, but we provide details for the reader's convenience.

\begin{theorem}
\label{fsur}The homomorphism $\varphi:\mathcal{D}_{A}(D)\rightarrow
\mathcal{F}_{A}(D)$ is surjective, and $\ker\varphi$ $=\{d_{\alpha,-\alpha}%
\mid\alpha\in A\}$.
\end{theorem}

\begin{proof}
Suppose $f\in\mathcal{F}_{A}(D)$, and $F_{0}$ is some face of $D$. Choose an
arbitrary element $\alpha_{0}\in A$, and define $d(F_{0})=\alpha_{0}$. For any
other face $F$ of $D$, choose a smooth path $P$ from a point in $F_{0}$ to a
point in $F$. We may presume that $P$ does not come close enough to any
crossing to intersect any of the short segments removed to indicate
undercrossings, and we may also presume that there are only finitely many
intersections between $P$ and $D$. Suppose that as we follow $P$ from $F_{0}$
to $F$, we encounter faces and arcs in the order $F_{0},a_{1}^{\prime}%
,F_{1}^{\prime},\dots,a_{k}^{\prime},F_{k}^{\prime}=F$. Then we define%
\begin{equation}
\label{defphi}
d(F)=(-1)^{k}\alpha_{0}+\sum_{i=1}^{k}(-1)^{k-i}f(a_{i}^{\prime}).
\end{equation}

It turns out that $d(F)$ is independent of the path $P$. To see why, suppose
$P^{\prime}$ is some other smooth path from $F_{0}$ to $F$. Then $P$ can be
smoothly deformed into $P^{\prime}$. When a smooth deformation does not
involve any crossing of $D$, there is no effect on $d(F)$. When the
deformation passes through a crossing, the effect is to replace one passage of
an arc of $D$ with three passages. For instance, in Figure \ref{colorf1} we
might replace a passage across $a_{1}$ from $F_{1}$ to $F_{2}$ with a sequence
of three passages; the first from $F_{1}$ to $F_{4}$ across $a_{3}$, the
second from $F_{4}$ to $F_{3}$ across $a_{2}$, and the third from $F_{3}$ to
$F_{2}$ across $a_{3}$. Suppose the original passage from $F_{1}$ to $F_{2}$
was indexed with $F_{1}=F_{j}^{\prime}$, $a_{1}=a_{j+1}^{\prime}$ and
$F_{2}=F_{j+1}^{\prime}$. The value given by (\ref{defphi}) is not changed
because Definition \ref{foxc} tells us that
\begin{align*}
&  (-1)^{k-j}f(a_{3})+(-1)^{k-j-1}f(a_{2})+(-1)^{k-j-2}f(a_{3})\\
&  =(-1)^{k-j}\cdot(-f(a_{2})+2\cdot f(a_{3}))=(-1)^{k-j}f(a_{1}).
\end{align*}
The same kind of argument holds if any other one of the passages in Figure
\ref{colorf1} is replaced with the remaining three.

By the way, \textquotedblleft replacing one passage with
three\textquotedblright\ might seem to indicate that the path is getting
longer. This is not true in general because some of the three new passages may
be canceled if $P$ includes the opposite passages through the same crossings. Also, if $i<j$ and $F_{i}^{\prime}=F_{j}^{\prime}$ then the argument above shows that the value of (\ref{defphi}) is not changed if $P$ is shortened to a path corresponding to the list $F_{0},a_{1}^{\prime},F_{1}^{\prime},\dots,a_{i}^{\prime},F_{i}^{\prime},a_{j+1}^{\prime},\dots,F_{k}^{\prime}=F$.

To verify that $d\in\mathcal{D}_{A}(D)$, suppose the $d$ values of the
four faces indicated in Figure \ref{colorf1} are defined using a path $P$ from
$F$ that enters the figure in the lower left hand corner. Then equation
(\ref{defphi}) implies that $d(F_{2})=f(a_{1})-d(F_{1})$, $d(F_{3})=f(a_{3}%
)-d(F_{2})$ and $d(F_{4})=f(a_{3})-d(F_{1})$. It follows that
\[
d(F_{1})+d(F_{4})=f(a_{3})=d(F_{2})+d(F_{3}).
\]

As $\varphi(d)=f$, we have verified that $\varphi$ is surjective. The description of $\ker\varphi$ in the statement is obvious.
\end{proof}

It is easy to see that the epimorphism $\varphi$ splits.

\begin{theorem}
\label{split}If $D$ is a link diagram and $A$ is an abelian group then $\mathcal{D}_{A}(D)$ is the internal direct sum of $\ker\varphi$ and a subgroup isomorphic to $\mathcal{F}_{A}(D)$.
\end{theorem}

\begin{proof}
Let $F_{0}$ be a fixed but arbitrary face of $D$. Let $\delta
:\mathcal{F}_{A}(D)\rightarrow \mathcal{D}_{A}(D)$ be the function defined by the construction in the proof of Theorem 4, always using $\alpha_0=0$. Then formula (\ref{defphi}) tells us that $\delta$ is a homomorphism.

As $\varphi \delta(f)=f$ $\forall f \in \mathcal{F}_{A}(D)$, the theorem follows.
\end{proof}

\section{The Goeritz matrix}

Suppose $D$ is a link diagram, and $s$ is one of the two checkerboard colorings of its faces. Then each crossing of $D$ is assigned a \emph{Goeritz index} $\eta \in \{-1,1\}$ as indicated in Figure \ref{colorf2}. That is, $\eta=1$ if the overpassing arc appears on the right-hand sides of the unshaded face(s) incident at the crossing, and $\eta=-1$ if the overpassing arc appears on the left-hand sides of the unshaded face(s) incident at the crossing. Equivalently, if we index faces as in Figure \ref{colorf1} then $\eta=(-1)^{\sigma(F_{4})}$.

\begin{figure}[h]%
\centering
\includegraphics[
trim=2.541393in 7.625780in 3.214340in 1.610119in,
height=1.3534in,
width=2.0851in
]%
{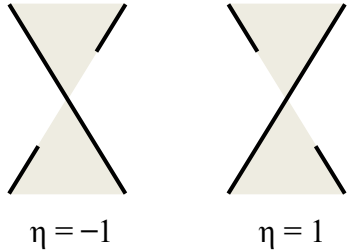}%
\caption{The Goeritz index of a crossing.}%
\label{colorf2}%
\end{figure}

\begin{definition}
\label{goeritz}Let $D$ be a link diagram, and let $s$ be either of the two checkerboard shadings of the faces of $D$. Let $F_{1},\dots,F_{n}$ be the unshaded faces of $D$. For $i,j\in\{1,\dots,n\}$ let $C_{ij}$ be the set of
crossings of $D$ incident on $F_{i}$ and $F_{j}$. Then the \emph{unreduced
Goeritz matrix} of $D$ with respect to $s$ is the $n \times n$ matrix $G(D,s)$ with entries defined as follows.
\[
G(D,s)_{ij}=%
\begin{cases}
-\sum\limits_{c\in C_{ij}}\eta(c)\text{,} & \text{if }i\neq j\\
-\sum\limits_{k\neq i}G(D,s)_{ik}\text{,} & \text{if }i=j
\end{cases}
\]
\end{definition}

Before proceeding we make five remarks about Definition \ref{goeritz}. (i) $G(D,s)$ is a symmetric integer matrix, whose rows and columns sum to $0$. It is traditional to remove one row and the corresponding column, to obtain a matrix whose determinant might not be $0$. However we do not follow this tradition here; that is why we call our matrix \textquotedblleft unreduced.\textquotedblright\ (ii) $G(D,s)$ ignores any crossing that is incident on only one unshaded face. (iii) There are two~\emph{checkerboard graphs} or~\emph{Tait graphs} associated with $D$. One graph has vertices corresponding to the shaded faces, and the other graph has vertices corresponding to the unshaded faces. Both have edges corresponding to the crossings of $D$. The matrix $G(D,s)$ is the~\emph{Laplacian matrix} of the unshaded checkerboard graph, with the Goeritz indices interpreted as edge weights. There is a well developed theory of Laplacian matrices of weighted graphs; the interested reader might consult~\cite[Chapter 13]{GR} for an introduction. (iv) If $\overline{s}$ is the other checkerboard shading of $D$ then the matrices $G(D,s)$ and $G(D,\overline{s})$ may be quite different; for instance, one may be much larger than the other. Nevertheless there is an intimate relationship between the two matrices. See \cite{LW} for a discussion. (v) Despite the connection between Goeritz matrices of link diagrams and Laplacian matrices of graphs, we have chosen to use relatively little terminology from graph theory in this paper. One reason for this choice is that the definition of the dual of a plane graph always results in a connected graph. In contrast, the checkerboard graphs of a split link diagram may be disconnected.

We use $A^{n}$ to denote the direct sum of $n$ copies of the abelian group $A$. If $F_{1},\dots,F_{n}$ are the unshaded faces of $D$ then $G(D,s)$ defines a homomorphism $A^{n} \to A^{n}$ of abelian groups. We are interested in the properties of the kernel of this homomorphism.

\begin{definition}
\label{kerdef}
Let $D$ be a link diagram with a shading $s$ whose unshaded faces are $F_{1},\dots,F_{n}$. If $A$ is
an abelian group then
\[
\ker_{A}G(D,s)=\{v\in A^{n}\mid G(D,s)\cdot v=0\}.
\]
\end{definition}

That is, $\ker_{A}G(D,s)$ is the subset of $A^n$ consisting of elements that are orthogonal to the rows of $G(D,s)$.

The next proposition is concerned with a special property of some split link diagrams: they have unshaded faces whose boundaries are not connected.

\begin{proposition}
\label{kernel}Let $D$ be a link diagram with a shading $s$ whose unshaded faces are $F_{1},\dots,F_{n}$. Let $\gamma$ be a simple closed curve, which forms part of the boundary of $F_{i}$. Let $\rho(\gamma)\in\mathbb{Z}^{n}$ be the vector defined as follows.
\[
\rho(\gamma)_{j}=%
\begin{cases}
-\sum\limits_{c\in C_{ij}\cap\gamma}\eta(c)\text{,} & \text{if }i\neq j\text{
}\\
-\sum\limits_{k\neq i}\rho(\gamma)_{k}\text{,} & \text{if }i=j\text{ }%
\end{cases}
\]
Then for any abelian group $A$, $\rho(\gamma)\cdot v=0$ $\forall v\in\ker_{A}G(D,s)$.
\end{proposition}

\begin{proof}
If $\gamma$ is the entire boundary of $F_{i}$ then $\rho(\gamma)$ is the
$i^{\text{th}}$ row of $G(D,s)$. If $\gamma$ is not incident on any crossing
that involves an unshaded face other than $F_{i}$, then $\rho(\gamma)=0$. In
either of these cases it is obvious that $\rho(\gamma)\cdot v=0$ $\forall
v\in\ker_{A}G(D,s)$.

Suppose $\gamma$ is not the entire boundary of $F_{i}$, and $\gamma$ is
incident on some crossing that involves an unshaded face other than $F_{i}$.
According to the Jordan curve theorem, $\mathbb{R}^{2}-\gamma$ has two
components, one inside $\gamma$ and the other outside $\gamma$. Suppose the
other unshaded face that shares a crossing on $\gamma$ with $F_{i}$ lies inside $\gamma$. (See Figure \ref{colorf6} for an example
of this sort. In the figure, $\gamma$ is indicated with dashes; it is displaced a little bit for clarity.) Let $D^{\prime}$ be the subdiagram of $D$ that includes $\gamma$ and all the arcs of $D$ contained inside $\gamma$, and let $s'$ be the shading of the faces of $D'$ defined by $s$. Then $F_{i}$ corresponds to an unshaded face of
$D^{\prime}$, whose boundary in $D^{\prime}$ is $\gamma$. The other unshaded faces of $D^{\prime}$ are the unshaded faces of $D$ contained inside
$\gamma$, and if $F_{j}$ is an unshaded face of $D$ contained inside $\gamma$
then the only difference between the row of $G(D,s)$ corresponding to $F_{j}$
and the row of $G(D^{\prime},s')$ corresponding to $F_{j}$ is that the former has
extra $0$ entries in columns corresponding to unshaded faces of $D$ not contained inside $\gamma$.

\begin{figure}[t]%
\centering
\includegraphics[
trim=1.604196in 7.097512in 1.740145in 1.074146in,
height=2.1525in,
width=3.8934in
]%
{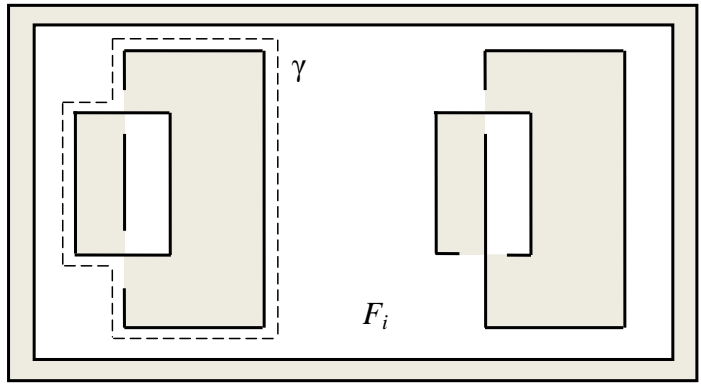}%
\caption{The boundary of $F_{i}$ consists of three closed curves.}%
\label{colorf6}%
\end{figure}

That is, if $G^{\prime}$ is the submatrix of $G(D^{\prime},s')$ obtained by
removing the row corresponding to the face of $D^{\prime}$ that contains $F_i$, then
\[
G(D,s)=%
\begin{pmatrix}
G^{\prime} & 0\\
G^{\prime\prime} & G^{\prime\prime\prime}%
\end{pmatrix}
\]
for some submatrices $G^{\prime\prime}$ and $G^{\prime\prime\prime}$.
Definition \ref{goeritz} makes it clear that the sum of the rows of a Goeritz
matrix is $0$; hence we can obtain the row of $G(D^{\prime},s')$ corresponding to
the face of $D^{\prime}$ that contains $F_i$ by summing the other rows of $G(D^{\prime
},s')$, and multiplying by $-1$. It follows that $-\rho(\gamma)$ is the sum of
the rows of $%
\begin{pmatrix}
G^{\prime} & 0
\end{pmatrix}
$, so $\rho(\gamma)$ is an element of the row space of $G(D,s)$. We conclude
that $\rho(\gamma)\cdot v=0$ for every $v\in\ker_{A}G(D,s)$.

If the other unshaded face that shares a crossing on $\gamma$ with $F_{i}$ lies outside $\gamma$ then the same argument applies, with
\textquotedblleft inside $\gamma$\textquotedblright\ changed to \textquotedblleft
outside $\gamma$\textquotedblright\ throughout.
\end{proof}

\begin{proposition}
\label{kertwo}Let $D$ be a link diagram with a shading $s$ whose unshaded faces are $F_{1},\dots,F_{n}$. Suppose $A$ is an abelian group, $v=(v_{1},\dots,v_{n})\in
\ker_{A}G(D,s)$, and $F_{i}$ and $F_{j}$ are incident at a crossing where only
one shaded face is incident. Then $v_{i}=v_{j}$.
\end{proposition}

\begin{proof}
Suppose $F_{i}$ and $F_{j}$ are incident at a crossing $c$, and $S$ is the
only shaded face of $D$ incident at $c$. Then there is a piecewise smooth closed curve
$\lambda$ that is contained in the interior of $S$ except for the fact that it
passes through $c$. Interchanging $i$ and $j$ if necessary, we may presume
that $F_{i}$ is contained in the region inside $\lambda$ and $F_{j}$ is
contained in the region outside $\lambda$. Let $D^{\prime}$ be the link diagram
obtained from $D$ by smoothing $c$ in such a way that $F_i$ is detached from $F_j$. (See Figure~\ref{colorf4a} for an example.) 

\begin{figure}[t]%
\centering
\includegraphics[
trim=2.408843in 7.893216in 2.409692in 1.607918in,
height=1.1563in,
width=2.7856in
]%
{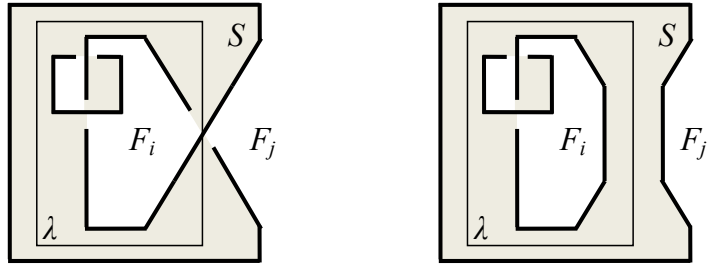}%
\caption{Examples of $D$ and $D^{\prime}$ in Proposition~\ref{kertwo}.}
\label{colorf4a}%
\end{figure}

Let $D^{\prime\prime}$ be the subdiagram of $D^{\prime}$
contained inside $\lambda$, and $D^{\prime\prime\prime}$ the subdiagram outside
$\lambda$. If $s',s''$ and $s'''$ are the shadings of $D',D''$ and $D'''$ defined by $s$ then the Goeritz matrix of $D^{\prime}$ is%
\[
G(D^{\prime},s')=%
\begin{pmatrix}
G(D^{\prime\prime},s'') & 0\\
0 & G(D^{\prime\prime\prime},s''')
\end{pmatrix}
.
\]
Let $G(D,s)_{i},G(D',s')_{i}$ and $G(D'',s'')_{i}$ be the rows of $G(D,s),G(D',s')$ and $G(D'',s'')$ corresponding to $F_{i}$, respectively. Also, let $G^{\prime\prime}$ be the submatrix of $G(D^{\prime\prime},s'')$ obtained by
removing $G(D^{\prime\prime},s'')_{i}$. Then $H=%
\begin{pmatrix}
G^{\prime\prime} & 0
\end{pmatrix}
$ is a submatrix of $G(D,s)$, and the sum of the rows of $H$ is the negative of $G(D^{\prime},s')_{i}$. Let $w=(w_{1},\dots
,w_{n})$ be the vector whose only nonzero entries are $w_{i}=\eta(c)$ and
$w_{j}=-\eta(c)$. Then $w$ is the difference between $G(D,s)_{i}$ and $G(D^{\prime},s')_{i}$. It follows that if we add $G(D,s)_{i}$ to the sum of the rows of $H$, we get $w$.

We conclude that $w$ is included in the row space of $G(D,s)$, so $w\cdot v=0$
$\forall v\in\ker_{A}G(D,s)$.
\end{proof}

\section{$\mathcal{D}_{A}(D)$ and $\ker_{A}G(D,s)$}

In this section we generalize ideas of Lamey, Silver and Williams \cite{LSW} to the Goeritz matrix. The foundation for this generalization has already been laid: Definitions~\ref{foxc} and \ref{dehnc} allow for colorings with values in arbitrary abelian groups, Definition~\ref{goeritz} allows for link diagrams with arbitrary crossing signs, and Propositions~\ref{kernel} and~\ref{kertwo} provide useful special properties of link diagrams with disconnected checkerboard graphs. Extending the arguments of~\cite{LSW} to the general setting requires only a little attention to special cases, and one additional idea given in Definition~\ref{beta}. 

\begin{proposition}
\label{kerone}Let $s$ be a shading of a link diagram $D$, whose unshaded faces are $F_{1},\dots,F_{n}$. If $d\in\mathcal{D}_{A}(D)$, then the vector
$v(d)=(d(F_{1}),\dots,d(F_{n}))$ is an element of $\ker_{A}G(D,s)$.
\end{proposition}

\begin{proof}
Suppose $1\leq i\leq n$, and let $G(D,s)_{i}$ be the $i^{\text{th}}$ row of
$G(D,s)$. The dot product $G(D,s)_{i}\cdot v(d)$ is a sum of contributions from
the crossings of $D$ incident on $F_{i}$. To analyze these contributions,
consider a crossing of $D$ incident on $F_{i}$ and $F_{j}\neq F_{i}$, pictured
in Figure \ref{colorf3}.%

\begin{figure}[htbp]
\centering
\includegraphics[
trim=2.407993in 7.627981in 2.276293in 1.609019in,
height=1.3543in,
width=2.8885in
]%
{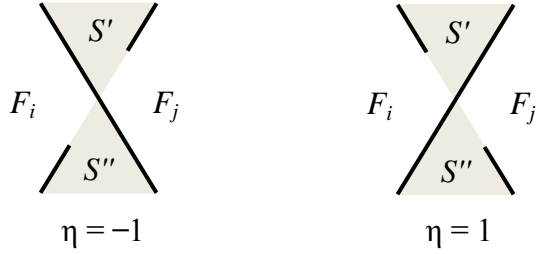}%
\caption{A crossing incident on $F_{i}$ and $F_{j}$.}%
\label{colorf3}%
\end{figure}

If $\eta=-1$, then the contribution of this crossing to $G(D,s)_{i}\cdot v(d)$
includes two terms: $-\eta\cdot d(F_{j})=d(F_{j})$ from $G(D,s)_{ij}\cdot d(F_{j})$ and
$\eta\cdot d(F_{i})=-d(F_{i})$ from $G(D,s)_{ii}\cdot d(F_{i})$. Consulting
Definition \ref{dehnc}, we see that the contribution of this crossing is%
\[
d(F_{j})-d(F_{i})=d(S^{\prime})-d(S^{\prime\prime}).
\]
If $\eta=1$, the contribution of this crossing to $G(D,s)_{i}\cdot v(d)$
includes $-\eta\cdot d(F_{j})=-d(F_{j})$ from $G(D,s)_{ij}\cdot d(F_{j})$ and
$\eta\cdot d(F_{i})=d(F_{i})$ from $G(D,s)_{ii}\cdot d(F_{i})$. Consulting
Definition \ref{dehnc}, we see that the contribution of this crossing is%
\[
-d(F_{j})+d(F_{i})=d(S^{\prime})-d(S^{\prime\prime}).
\]
Either way, the contribution is $d(S^{\prime})-d(S^{\prime\prime})$.

Now, consider a crossing as pictured in Figure \ref{colorf3}, but with $i=j$.
This crossing is ignored by $G(D,s)$, so its contribution to $G(D,s)_{i}\cdot
v(d)$ is $0$. On the other hand, Definition \ref{dehnc} tells us that
$F_{i}=F_{j}\implies d(S^{\prime})=d(S^{\prime\prime})$, so $d(S^{\prime
})-d(S^{\prime\prime})$ is $0$ too.

In every case, then, the contribution of the crossing pictured in Figure
\ref{colorf3} to $G(D,s)_{i}\cdot v(d)$ is $d(S^{\prime})-d(S^{\prime\prime})$.
If we follow the boundary component of $F_{i}$ that contains this crossing in the clockwise direction, we see that the total of the contributions of all the crossings is a telescoping sum:
\[
(d(S^{\prime})-d(S^{\prime\prime}))+(d(S^{\prime\prime})-d(S^{\prime
\prime\prime}))+\cdots+(d(S^{(k)})-d(S^{\prime}))=0\text{,}%
\]
where $k$ is the number of crossings incident on this boundary component of $F_{i}$. As every boundary component of $F_i$ contributes $0$, $G(D,s)_{i}\cdot v(d)=0$.
\end{proof}

Proposition \ref{kerone} tells us that $d\mapsto v(d)$ defines a map
$v:\mathcal{D}_{A}(F)\rightarrow\ker_{A}G(D,s)$. It is easy to see that $v$ is a
homomorphism of abelian groups, and that $\ker v$ consists of the Dehn
colorings that are identically $0$ on unshaded faces of $D$. It is more
difficult to see another important property of $v$: it is surjective.

\begin{proposition}
\label{kerthree}Let $D$ be a link diagram, with a shading $s$ whose unshaded faces are $F_{1}%
,\dots,F_{n}$. Let $A$ be an abelian group, and suppose $v\in\ker_{A}G(D,s)$. Then there is a $d\in\mathcal{D}_{A}(D)$ with $v=v(d)$.
\end{proposition}

\begin{proof}
For each unshaded face $F_{i}$ of $D$, define $d(F_{i})$ to be the
$i^{\text{th}}$ coordinate of the vector $v$. Our job is to define $d$ values
for the shaded faces, in such a way that the resulting function
$d:F(D)\rightarrow A$ satisfies Definition \ref{dehnc}.

Choose any shaded face $S$ of $D$, choose any element $\alpha\in A$, and
define $d(S)=\alpha$. Repeat the following recursive step as many times as
possible. If $S^{\prime\prime}$ is a shaded face such that $d(S^{\prime\prime
})$ has not yet been defined, and the boundary of $S^{\prime\prime}$ shares a
crossing with the boundary of a shaded face $S^{\prime}$ such that
$d(S^{\prime})$ has been defined, then define $d(S^{\prime\prime})$ to be the
unique element of $A$ that satisfies Definition \ref{dehnc} at this crossing.

After the process of the preceding paragraph is completed, there may still be
shaded faces whose $d$ values have not been defined; these faces do not share
any crossing with shaded faces whose $d$ values have been defined. Go back to
the preceding paragraph, and change the first sentence to read
\textquotedblleft Choose any shaded face $S$ of $D$ whose $d$ value has not
yet been defined, choose any element $\alpha\in A$, and define $d(S)=\alpha
$.\textquotedblright\ Then repeat the entire process of the preceding
paragraph as many times as possible.

We claim that this recursion yields a well defined function $d\in
\mathcal{D}_{A}(D)$. The claim certainly holds in case no crossing appears in
$D$ as $G(D,s)$ is the $0$ matrix, all vectors with entries in $A$ have
$G(D,s)\cdot v=0$, and all functions $F(D)\rightarrow A$ lie in $\mathcal{D}%
_{A}(D)$. We proceed with the assumption that there is at least one crossing
in $D$.

Let $S$ be a shaded face of $D$. If $S$ does not share a crossing of $D$ with
any other shaded face then the value of $d(S)$ is handled by the first
sentence of the second paragraph (as modified in the third paragraph). In this
case it is obvious that $d(S)$ is well defined. Proposition \ref{kertwo} tells us
that $v_{i}=v_{j}$ whenever $F_{i}$ and $F_{j}$ share a crossing of $D$ with
$S$, so $d$ satisfies Definition \ref{dehnc} at all crossings involving $S$.
If all shaded faces of $D$ are of this type, we are done.

The rest of the proof resembles the proof of Theorem \ref{fsur}. Suppose $S$
is a shaded face such that $d(S)$ is defined through the process of the second
paragraph, by applying Definition \ref{dehnc} at a crossing where $S$ and
another shaded face are incident. Let $S_{0},c_{1},S_{1},\dots,c_{k},S_{k}=S$
be the sequence of shaded faces and crossings that was used to determine the
value of $d(S)$, with $S_{0}$ handled by the first sentence of the second
paragraph. This sequence corresponds to a piecewise smooth path $P$ from a point in the
interior of $S_{0}$ near $c_{1}$ to a point in the interior of $S$ near
$c_{k}$, which stays inside shaded faces except when it passes through
crossings. Any other possible recursive definition of $d(S)$ corresponds to a
similar path $P^{\prime}$ from $S_{0}$ to $S$, which can be deformed into $P$
by some sequence of two types of moves: the type illustrated in\ Figure
\ref{colorf5}, and the trivial type in which two consecutive passages in
opposite directions through the same crossing are canceled.

\begin{figure}[tbp]%
\centering
\includegraphics[
trim=3.614539in 7.360545in 2.142893in 1.075247in,
height=1.9545in,
width=2.0816in
]%
{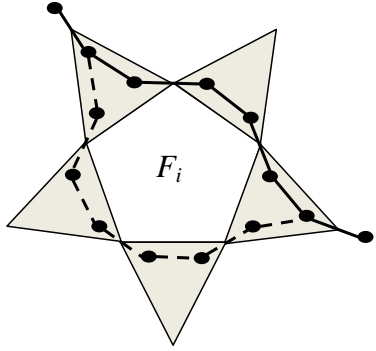}%
\caption{Deforming $P$.}%
\label{colorf5}%
\end{figure}

To verify that $d(S)$ is well defined, we show that the deformation
illustrated in Figure \ref{colorf5} does not change the value of $d(S)$. Let
$S_{p},c_{p+1},S_{p+1},\dots,c_{q},S_{q}$ be the portion of the sequence
$S_{0},c_{1},S_{1},\dots,c_{k},S_{k}$ that is a sequence of shaded neighbors
of $F_{i}$. Then $P^{\prime}$ is obtained by replacing this portion with
$S_{p}=S_{p}^{\prime},c_{p+1}^{\prime},S_{p+1}^{\prime},\dots,c_{q^{\prime}%
}^{\prime},S_{q^{\prime}}^{\prime}=S_{q}$, where the crossings $c_{p+1}%
,\dots,c_{q}$, $c_{q+1}^{\prime},\dots,c_{p+1}^{\prime}$ appear in this order
on a closed curve $\gamma$ contained in the boundary of $F_{i}$. Interchanging
the names of $P$ and $P^{\prime}$ if necessary, we may presume that
$c_{p+1},\dots,c_{q}$, $c_{q+1}^{\prime},\dots,c_{p+1}^{\prime}$ appear in
this order clockwise around $\gamma$. For $1\leq j\leq
q-p$ let $U_{j}$ be the unshaded face that shares the crossing $c_{p+j}$ with
$F_{i}$, and for $1\leq j\leq q^{\prime}-p$ let $U_{j}^{\prime}$ be the
unshaded face that shares the crossing $c_{p+j}^{\prime}$ with $F_{i}$. Then
when the second paragraph states that at each step the unique value of
$d(U_{j})$ or $d(U_{j}^{\prime})$ that satisfies Definition \ref{dehnc} is
used, it means that for each $j\geq1$,
\begin{align*}
d(S_{p+j})  &  =d(S_{p+j-1})+\eta(c_{p+j})\cdot(d(U_{j})-d(F_{i}))\\
&  =d(S_{p+j-1})+\eta(c_{p+j})\cdot(v(U_{j})-v(F_{i}))
\end{align*}%
\begin{align*}
\text{and }d(S_{p+j}^{\prime})  &  =d(S_{p+j-1}^{\prime})-\eta(c_{p+j}%
^{\prime})\cdot(d(U_{j}^{\prime})-d(F_{i}))\\
&  =d(S_{p+j-1}^{\prime})-\eta(c_{p+j}^{\prime})\cdot(v(U_{j}^{\prime
})-v(F_{i})).
\end{align*}
When we follow $P$ we conclude that
\[
d(S_{q})-d(S_{p})=\sum_{j=1}^{q-p}(d(S_{p+j})-d(S_{p+j-1}))=\sum_{j=1}%
^{q-p}\eta(c_{p+j})\cdot(v(U_{j})-v(F_{i}))
\]
and when we follow $P^{\prime}$, we conclude that%
\[
d(S_{q^{\prime}}^{\prime})-d(S_{p})=\sum_{j=1}^{q^{\prime}-p}(d(S_{p+j}%
^{\prime})-d(S_{p+j-1}^{\prime}))=-\sum_{j=1}^{q^{\prime}-p}\eta
(c_{p+j}^{\prime})\cdot(v(U_{j}^{\prime})-v(F_{i})).
\]
The difference between these two values is
\begin{gather*}
-\sum_{j=1}^{q-p}\eta(c_{p+j})\cdot(v(U_{j})-v(F_{i}))-\sum_{j=1}^{q^{\prime
}-p}\eta(c_{p+j}^{\prime})\cdot(v(U_{j}^{\prime})-v(F_{i}))\\
=\left(  \sum_{j=1}^{q-p}\eta(c_{p+j})+\sum_{j=1}^{q^{\prime}-p}\eta
(c_{p+j}^{\prime})\right)  \cdot v(F_{i})\\
-\sum_{j=1}^{q-p}\eta(c_{p+j})\cdot v(U_{j})-\sum_{j=1}^{q^{\prime}-p}%
\eta(c_{p+j}^{\prime})\cdot v(U_{j}^{\prime})=\rho(\gamma)\cdot v\text{,}%
\end{gather*}
where $\rho(\gamma)$ is the vector discussed in Proposition \ref{kernel}. As
$\rho(\gamma)\cdot v=0$, $P$ and $P^{\prime}$ lead to the same value for
$d(S)$.

It remains to verify that $d\in\mathcal{D}_{A}(D)$. Suppose $c$ is a crossing
of $D$, as pictured in Figure \ref{colorf1}. If a single shaded face $S$
appears twice in the figure then Definition \ref{dehnc} requires that the two
unshaded faces $F_{i},F_{j}$ in the figure have the same $d$ value.
Proposition \ref{kertwo} assures us that this is the case. If two different
shaded faces appear, then the well-definedness of $d$ allows us to assume
that the $d$ value of one of the two shaded faces is calculated directly
from the $d$ value of the other one, as in the third sentence of the second
paragraph of the proof. But this calculation is performed precisely to
guarantee that $d$ satisfies Definition \ref{dehnc} at the crossing $c$.
\end{proof}

Combining Propositions \ref{kerone} and \ref{kerthree}, we obtain the following.

\begin{theorem}
Let $D$ be a link diagram with a shading $s$, and $A$ an abelian group. Then $v:\mathcal{D}_{A}(D)\rightarrow\ker_{A}G(D,s)$ is a surjective homomorphism, whose kernel consists of the Dehn colorings that are identically $0$ on unshaded faces of $D$. 
\end{theorem}

The next definition allows us to give a precise description of $\ker v$.

\begin{definition}
\label{beta}
Let $D$ be a link diagram with a shading $s$. Let $\Gamma_s(D)$ denote the shaded checkerboard graph of $D$, i.e., $\Gamma_s(D)$ has a vertex for each shaded face of $D$ and an edge for each crossing of $D$, with the edge corresponding to a crossing $c$ incident on the vertex (or vertices) corresponding to the shaded face(s) incident at $c$. Then $\beta_s(D)$ denotes the number of connected components of $\Gamma_s(D)$.
\end{definition}

If $d \in \ker v$ then $d$ is identically $0$ on unshaded faces of $D$, so Definition \ref{dehnc} is equivalent to the requirement that $d(S)=d(S')$ whenever $S$ and $S'$ are shaded faces of $D$ incident at the same crossing. It follows that $d$ is constant on each connected component of $\Gamma_s(D)$, and these constant values are arbitrary. We deduce that 
\[
\ker v \cong A^{\beta_s(D)}.
\]

\section{The theorem of Nanyes}

In this section we discuss two versions of the theorem of Nanyes \cite{N}, one involving Dehn colorings and the other involving Fox colorings. The first version asserts that like $\varphi:\mathcal{D}_{A}(D)\rightarrow\mathcal{F}_{A}(D)$, the epimorphism $v:\mathcal{D}_{A}(D)\rightarrow\ker_{A}G(D,s)$ splits. 

\begin{theorem}
\label{main2}Let $D$ be a link diagram, and $A$ an abelian group. Then $\mathcal{D}_{A}(D)$ is the internal direct sum of $\ker v$ and a subgroup isomorphic to $\ker_{A}G(D,s)$.
\end{theorem}

\begin{proof}
Choose shaded faces $S_{1},\dots,S_{\beta_s(D)}$ of $D$, one in each connected
component of the graph $\Gamma_s(D)$. Let $u:\ker_{A}G(D,s) \to \mathcal{D}_{A}(D)$ be the map defined by the construction in the proof of Proposition \ref{kerthree}, always using one of $S_{1},\dots,S_{\beta_s(D)}$ when implementing the first sentence of the second paragraph, and always using $0$ for the value of $d(S_i)$. The uniqueness of the $d$ values calculated in the other steps of the recursion guarantees that $u$ is a homomorphism. As $vu$ is the identity map of $\ker_{A}G(D,s)$, the theorem follows.
\end{proof}

Theorem \ref{main2} implies that 
\[
\mathcal{D}_{A}(D) \cong \ker_{A}G(D,s) \oplus A^{\beta_s(D)}
.
\]
The second version of the theorem of Nanyes \cite{N} is the corresponding description of $\mathcal{F}_{A}(D)$ up to isomorphism. (N.b.\ If $A$ is an abelian group then $A^0$  denotes $\{0\}$, the abelian group with only one element.)

\begin{theorem}
\label{nanyes} Let $D$ be a link diagram, and $A$ an abelian group. Then%
\[
\mathcal{F}_{A}(D)\cong\ker_{A}G(D,s)\oplus A^{\beta_s(D)-1}\text{.}%
\]
\end{theorem}

\begin{proof}
Choose a shaded face $S_{1}$ of $D$, and let 
\[
\mathcal{D}_{1}=\{d \in \mathcal{D}_{A}(D) \mid d(S_{1})=0 \} \text{.}
\]

If we apply the construction in the proof of Theorem \ref{fsur} with $S_{1}$ always playing the role of $F_{0}$ and $0$ always playing the role of $\alpha_0$, we conclude that the restricted mapping  $(\varphi \mid \mathcal{D}_{1}):\mathcal{D}_{1} \to \mathcal{F}_{A}(D)$ is surjective. As $\ker \varphi = \{d_{\alpha,-\alpha} \mid \alpha \in A \}$, $\ker (\varphi \mid \mathcal{D}_{1})=\{0\}$. Hence $\mathcal{F}_{A}(D) \cong \mathcal{D}_{1}$.

If we apply the construction in the proof of Proposition \ref{kerthree} with $S_1$ always used in the first implementation of the first sentence of the second paragraph and $0$ always used as the arbitrarily chosen value of $d(S_1)$, then we conclude that the restricted mapping $(v \mid \mathcal{D}_{1}):\mathcal{D}_{1} \to \ker_{A}G(D,s)$ is surjective. The description of $\ker v$ at the end of the preceding section applies to $\ker (v \mid \mathcal{D}_{1})$, with the exception that for $d \in \ker (v \mid \mathcal{D}_{1})$ the value of $d$ on the connected component of $\Gamma_{s}(D)$ containing $S_1$ is not arbitrary; it is $0$. We deduce that $\ker (v \mid \mathcal{D}_{1}) \cong A^{\beta_s(D)-1}$. 

Applying the proof of Theorem \ref{main2} to $\mathcal{D}_{1}$, we conclude that $\mathcal{D}_{1}$ is the internal direct sum of $\ker (v \mid \mathcal{D}_{1})$ and a subgroup isomorphic to $\ker_{A}G(D,s)$. 
\end{proof}

Notice that checkerboard colorings are not mentioned in Definitions~\ref{foxc} and~\ref{dehnc}, but Theorems \ref{main2} and \ref{nanyes} tell us that $\mathcal{D}_{A}(D)$ and $\mathcal{F}_{A}(D)$ are determined up to isomorphism by the checkerboard graphs of $D$: the unshaded checkerboard graph provides $G(D,s)$, and the shaded checkerboard graph provides $\beta_s(D)$. Although the two checkerboard graphs play different roles, the theorems apply equally well if the shading $s$ is reversed.

We should explain that we refer to Theorems \ref{main2} and \ref{nanyes} as \textquotedblleft versions\textquotedblright\ of the theorem of Nanyes \cite{N} because there are several differences between our setup and Nanyes's: he required the image of a Fox coloring to generate $A$, he required a Fox coloring to be $0$ somewhere, and he removed a row and column from the Goeritz matrix. The first difference has the effect of shifting attention from the kernel to the cokernel of a homomorphism represented by $G(D,s)$. The second difference has the effect of removing a direct summand isomorphic to $A$ from $\mathcal{F}_{A}(D)$, and the third difference has the effect of removing another such direct summand from $\ker_{A}G(D,s)$. We leave further articulation of the details of the relationship between our discussion and that of \cite{N} to the reader.

We should also point out that Nanyes's descriptions of $\beta_s(D)$ and $G(D,s)$ in
\cite{N} are inaccurate. He described $G(D,s)$ as a matrix obtained using all
the faces of $D$, not just the unshaded faces; and he described $\beta_s(D)$ as
the number of connected components of the graph of unshaded faces, not
the graph of shaded faces. It is easy to see that either of these mistakes can lead to an erroneous description of $\mathcal{F}_{A}(D)$. For instance, let $D$ be a crossing-free diagram of a $\mu$-component unlink. Then $1 \leq \beta_s(D) \leq \mu$ and $D$ has $n=\mu +1 -\beta_s(D)$ unshaded faces. The corresponding Goeritz matrix has all entries $0$, so if we mistakenly replace $G(D,s)$ with a matrix that has a row and column for every face, we will conclude that $\ker_{A}G(D,s)\oplus A^{\beta_s(D)-1}$ is isomorphic to $A^{n+2\beta_s(D)-1}=A^{\mu + \beta_s(D)}$. If we use the correct definition of $G(D,s)$ then we obtain the $0$ matrix of order $n$, so $\ker_{A}G(D,s)\cong A^n$; if we then mistakenly calculate $\beta_s(D)$ using unshaded faces we will conclude that $\ker_{A}G(D,s)\oplus A^{\beta_s(D)-1}$ is isomorphic to $A^{2n-1}$. The correct calculation yields $\ker_{A}G(D,s)\oplus A^{\beta_s(D)-1} \cong A^{n+\beta_s(D)-1}=A^{\mu}$. This is isomorphic to $\mathcal{F}_{A}(D)$ because a Fox coloring of $D$ is simply a function that is constant on each component of $L(D)$.

\section{Two examples}

Let $T$ be the $(2,8)$ torus link diagram pictured on the left in Figure \ref{colorf7}, and $W$ the Whitehead link diagram pictured on the right. It is easy to see that $L(T)$ and $L(W)$ are inequivalent links: the linking number of the two components of $L(T)$ is $\pm4$, and the linking number of the two components of $L(W)$ is $0$. 

Nevertheless, Theorems \ref{main2} and \ref{nanyes} imply that every abelian group $A$ has $\mathcal{D}_{A}(T) \cong \mathcal{D}_{A}(W)$ and $\mathcal{F}_{A}(T) \cong \mathcal{F}_{A}(W)$. To see why, notice that according to Definition \ref{goeritz} the Goeritz matrices of $T$ and $W$ associated with the shadings of~Figure \ref{colorf7} are
\[
G(T,s)=
\begin{pmatrix}
-8 & 8\\
8 & -8%
\end{pmatrix}
\text{\quad and \quad}
G(W,s)=\begin{pmatrix}
-3 & 1 & 2\\
1 & -3 & 2\\
2 & 2 & -4
\end{pmatrix}
\text{.}
\]
Consequently, if $A$ is an abelian group then $\ker_{A}G(T,s)$ is the set of ordered pairs $(x,y) \in A^2$ such that $8y-8x=0$, and $\ker_{A}G(W,s)$ is the set of ordered triples $(a,b,c) \in A^3$ such that $-3a+b+2c=0$ and $a-3b+2c=0$. 

\begin{figure}[pbt]%
\centering
\includegraphics[
trim=1.605046in 7.763350in 1.339945in 1.342683in,
height=1.452in,
width=4.1926in
]%
{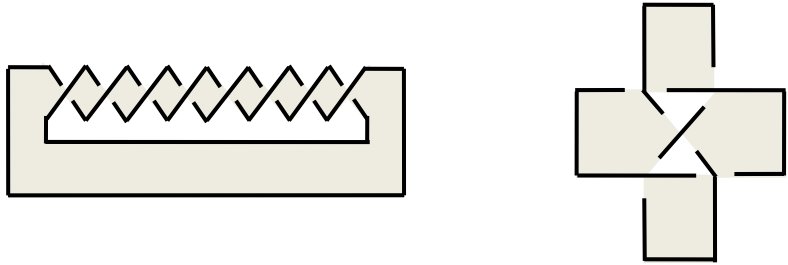}%
\caption{Diagrams of a $(2,8)$ torus link and a Whitehead link.}
\label{colorf7}%
\end{figure}

We claim that for every abelian group $A$, the formula $\pi(a,b,c) = (a,c)$ defines an isomorphism $\pi:\ker_{A}G(W,s) \to \ker_{A}G(T,s)$. The claim is verified in four steps. First, notice that if $(a,b,c) \in \ker_{A}G(W,s)$ then 
\[
8c-8a=3 \cdot (-3a+b+2c)+a-3b+2c=3 \cdot 0 + 0 = 0 \text{,}
\]
so $\pi(a,b,c) =(a,c)$ is an element of $\ker_{A}G(T,s)$. Second, notice that $\pi$ is a homomorphism because it is a restriction of the projection homomorphism $A^3 \to A^2$ defined by $(a,b,c) \mapsto (a,c)$. Third, notice that if $(x,y) \in \ker_{A}G(T,s)$ then $(x,3x-2y,y) \in \ker_{A}G(W,s)$, because
\[
-3x+(3x-2y)+2y=0 \quad \text{  and  } \quad x-3 \cdot (3x-2y)+2y=-8x+8y= 0 \text{.}
\]
As $\pi(x,3x-2y,y)=(x,y)$, we deduce that $\pi$ is surjective. Fourth, notice that if $(a,b,c),(a',b',c') \in \ker_{A}G(W,s)$ and $\pi(a,b,c)=\pi(a',b',c')$ then $a=a'$ and $c=c'$, so $b=3a-2c=3a'-2c'=b'$. We deduce that $\pi$ is injective.

If $\overline{s}$ denotes the shadings of $T$ and $W$ opposite to those indicated in Figure \ref{colorf7}, then $\beta_s(T)=\beta_s(W)=\beta_{\overline{s}}(T)= \beta_{\overline{s}}(W)=1$. According to Theorem \ref{nanyes}, it follows that $\ker_{A}G(T,s) \cong \ker_{A}G(T,\overline{s})$ and $\ker_{A}G(W,s) \cong \ker_{A}G(W,\overline{s})$. We leave it as an exercise for the reader to verify these isomorphisms directly.

\section{A special case}

Several authors have observed that when $A=\mathbb{Z}/2\mathbb{Z}$, there is
an especially simple relationship between $\ker_{A}G(D,s)$ and
the link $L(D)$~\cite{CG, H, LSW, N, SW, T}. In fact this simple relationship holds more generally when
the exponent of $A$ is $2,$ i.e., $2\cdot\alpha=0$ $\forall\alpha\in A$. We
close with a summary of this simple relationship.

Let $D$ be a link diagram, with $L(D)=K_{1}\cup\dots\cup K_{\mu}$. For each
$F\in F(D)$ there is a piecewise smooth path $P_{F}$ from a point in the
interior of $F$ to a point in the interior of the unbounded face of $D$, which does not come near any crossing of $D$ and has only a finite number of intersections with the arcs of $D$, all of which are transverse intersections. In general there are many such paths, with different patterns of intersections with $D$; but every such path will
have the same number of intersections (mod $2$) with the image of each $K_{i}$.

\begin{definition}
If $F\in F(D)$ then for $1\leq i\leq\mu$ the \emph{index} of $F$ with respect to $K_{i}$ is the parity (mod $2$) of the number of intersection points of a path $P_{F}$ with the image of $K_{i}$ in the plane. We denote the index $i_{D}(F,K_i)$.
\end{definition}

\begin{proposition}
Suppose $A$ is of exponent $2$. Then $f:A(D)\rightarrow A$ is a Fox coloring
of $D$ if and only if there is an element $(\alpha_{1},\dots
,\alpha_{\mu}) \in A^{\mu}$ such that $f(a)=\alpha_{i}$ whenever $a$ is an arc of $D$
that belongs to the image of $K_{i}$.
\end{proposition}

\begin{proof}
As $A$ is of exponent $2$, Definition \ref{foxc} simply requires
$f(a_{1})=f(a_{2})$ in Figure \ref{colorf1}. Applying this equation at every crossing, we conclude that $\mathcal{F}_{A}(D)$ is the set of maps $A(D) \to A$ that are constant on the image of every $K_{i}$.
\end{proof}

That is, $\mathcal{F}_{A}(D)$ may be identified naturally with $A^{\mu}$. It
follows from Theorem \ref{split} that $\mathcal{D}_{A}(D)$ may be identified
naturally with $A^{\mu+1}$. We spell out the details:

\begin{proposition}
\label{twodehn}
Suppose $A$ is of exponent $2$. Then $d:F(D)\rightarrow A$ is a Dehn coloring
of $D$ if and only if there is an element $(\alpha_{0},\dots,\alpha_{\mu
})\in A^{\mu +1}$ such that every $F\in F(D)$ has%
\[
d(F)=\alpha_{0}+ \sum_{i=1}^{\mu} \alpha_{i} \cdot i_{D}(F,K_i)
.
\]
\end{proposition}

\begin{proof}
The formula in the statement is the appropriate version of formula
(\ref{defphi}) from the proof of Theorem \ref{fsur}.
\end{proof}

\begin{corollary}
\label{gker}
Suppose $A$ is of exponent $2$, and $F_{1},\dots,F_{n}$ are the unshaded faces
of a shading $s$ of $D$. Then an element $v=(v_1,\dots,v_n)\in A^n$ is contained in $\ker_{A}G(D,s)$ if and only if there is an element $(\alpha_{0},\dots,\alpha_{\mu
})\in A^{\mu +1}$ such that %
\[
v_{j}=\alpha_{0}+\sum_{i=1}^{\mu} \alpha_{i} \cdot i_{D}(F_j,K_i)
\]
for every $j\in\{1,\dots,n\}$.
\end{corollary}

\begin{proof}
According to Proposition \ref{kerthree}, the elements of $\ker_{A}G(D,s)$ are
the elements of $A^{n}$ obtained by evaluating Dehn colorings on the unshaded
faces of $D$.
\end{proof}

The same formula appears in Proposition \ref{twodehn} and Corollary \ref{gker}, but the two maps defined by the formula are quite different. In Proposition \ref{twodehn}, the formula defines an isomorphism $A^{\mu +1} \to \mathcal{D}_{A}(D)$. In Corollary \ref{gker}, instead, the formula defines a split epimorphism $A^{\mu +1} \to \ker_{A}G(D,s)$, whose kernel is isomorphic to $A^{\beta_s(D)}$.

\bibliographystyle{plain}
\bibliography{colors}

\begin{thebibliography}{10}

\bibitem{A}
Colin~C. Adams.
\newblock {\em The knot book}.
\newblock American Mathematical Society, Providence, RI, 2004.
\newblock An elementary introduction to the mathematical theory of knots,
  Revised reprint of the 1994 original.

\bibitem{CSW}
J.~Scott Carter, Daniel~S. Silver, and Susan~G. Williams.
\newblock Three dimensions of knot coloring.
\newblock {\em Amer. Math. Monthly}, 121(6):506--514, 2014.

\bibitem{CG}
ZhiYun Cheng and HongZhu Gao.
\newblock On region crossing change and incidence matrix.
\newblock {\em Sci. China Math.}, 55(7):1487--1495, 2012.

\bibitem{C}
Richard~H. Crowell.
\newblock Nonalternating links.
\newblock {\em Illinois J. Math.}, 3:101--120, 1959.

\bibitem{CF}
Richard~H. Crowell and Ralph~H. Fox.
\newblock {\em Introduction to knot theory}.
\newblock Springer-Verlag, New York-Heidelberg, 1977.
\newblock Reprint of the 1963 original, Graduate Texts in Mathematics, No. 57.

\bibitem{GR}
Chris Godsil and Gordon Royle.
\newblock {\em Algebraic graph theory}, volume 207 of {\em Graduate Texts in
  Mathematics}.
\newblock Springer-Verlag, New York, 2001.

\bibitem{G}
Lebrecht Goeritz.
\newblock Knoten und quadratische {F}ormen.
\newblock {\em Math. Z.}, 36(1):647--654, 1933.

\bibitem{GL}
C.~McA. Gordon and R.~A. Litherland.
\newblock On the signature of a link.
\newblock {\em Invent. Math.}, 47(1):53--69, 1978.

\bibitem{H}
Megumi Hashizume.
\newblock On the homomorphism induced by region crossing change.
\newblock {\em JP J. Geom. Topol.}, 14(1):29--37, 2013.

\bibitem{ILL}
Young~Ho Im, Kyeonghui Lee, and Sang~Youl Lee.
\newblock Signature, nullity and determinant of checkerboard colorable virtual
  links.
\newblock {\em J. Knot Theory Ramifications}, 19(8):1093--1114, 2010.

\bibitem{J}
Fran\c{c}ois Jaeger.
\newblock On the {K}auffman polynomial of planar matroids.
\newblock In {\em Fourth {C}zechoslovakian {S}ymposium on {C}ombinatorics,
  {G}raphs and {C}omplexity ({P}rachatice, 1990)}, volume~51 of {\em Ann.
  Discrete Math.}, pages 117--127. North-Holland, Amsterdam, 1992.

\bibitem{Krem}
L.~H. {Kauffman}.
\newblock {Remarks on Formal Knot Theory}.
\newblock arxiv:math/0605622, May 2006.

\bibitem{Kau}
Louis~H. Kauffman.
\newblock {\em Knots and physics}, volume~53 of {\em Series on Knots and
  Everything}.
\newblock World Scientific Publishing Co. Pte. Ltd., Hackensack, NJ, fourth
  edition, 2013.

\bibitem{K}
R.~H. Kyle.
\newblock Branched covering spaces and the quadratic forms of links.
\newblock {\em Ann. of Math. (2)}, 59:539--548, 1954.

\bibitem{LSW}
Kalyn~R. Lamey, Daniel~S. Silver, and Susan~G. Williams.
\newblock Vertex-colored graphs, bicycle spaces and {M}ahler measure.
\newblock {\em J. Knot Theory Ramifications}, 25(6):1650033, 22, 2016.

\bibitem{Li}
W.~B.~Raymond Lickorish.
\newblock {\em An introduction to knot theory}, volume 175 of {\em Graduate
  Texts in Mathematics}.
\newblock Springer-Verlag, New York, 1997.

\bibitem{LW}
Magnhild Lien and William Watkins.
\newblock Dual graphs and knot invariants.
\newblock {\em Linear Algebra Appl.}, 306(1-3):123--130, 2000.

\bibitem{Lip}
A.~S. Lipson.
\newblock Link signature, {G}oeritz matrices and polynomial invariants.
\newblock {\em Enseign. Math. (2)}, 36(1-2):93--114, 1990.

\bibitem{L}
Charles Livingston.
\newblock {\em Knot theory}, volume~24 of {\em Carus Mathematical Monographs}.
\newblock Mathematical Association of America, Washington, DC, 1993.

\bibitem{N}
Ollie Nanyes.
\newblock Link colorability, covering spaces and isotopy.
\newblock {\em J. Knot Theory Ramifications}, 6(6):833--849, 1997.

\bibitem{P}
J\'ozef~H. Przytycki.
\newblock From {G}oeritz matrices to quasi-alternating links.
\newblock In {\em The mathematics of knots}, volume~1 of {\em Contrib. Math.
  Comput. Sci.}, pages 257--316. Springer, Heidelberg, 2011.

\bibitem{R}
K.~Reidemeister.
\newblock {\em Knotentheorie}.
\newblock Springer-Verlag, Berlin-New York, 1974.
\newblock Reprint.

\bibitem{S}
H.~Seifert.
\newblock Die {V}erschlingungsinvarianten der zyklischen
  {K}noten\"uberlager-ungen.
\newblock {\em Abh. Math. Sem. Univ. Hamburg}, 11(1):84--101, 1935.

\bibitem{SW}
Daniel~S. Silver and Susan~G. Williams.
\newblock On the component number of links from plane graphs.
\newblock {\em J. Knot Theory Ramifications}, 24(1):1520002, 5, 2015.

\bibitem{T}
Lorenzo Traldi.
\newblock On the {G}oeritz matrix of a link.
\newblock {\em Math. Z.}, 188(2):203--213, 1985.

\end{thebibliography}

\end{document}